\documentclass[reqno,twoside,11pt]{amsart}

\usepackage{dsfont}

\usepackage{verbatim}%to comment a paragraph \begin{comment}

\theoremstyle{plain}

\numberwithin{equation}{section}

\usepackage{graphicx}
\usepackage{amssymb}
\usepackage{amsmath}
\usepackage{latexsym}
\usepackage{mathrsfs}
\usepackage{textcomp}
\usepackage{fancyhdr}
\usepackage{pxfonts}
\usepackage{esint}
\usepackage{textcomp}
\usepackage{vmargin}
\usepackage{bbm}

\newcommand{\abs}[1]{\left\lvert #1 \right\rvert}
\newcommand{\dx}{\;\mbox{d}x}
\newcommand{\dxy}{\;\mbox{d}x_1\mbox{d}x_2}

\newcommand{\osc}[1]{\underset{#1}{\mathrm{osc}}\,}
\newcommand{\p}{\mathbbm{p}}

\theoremstyle{definition}

\newtheorem{Teo}{Theorem}[section]

\newtheorem{Lemma}[Teo]{Lemma}
\newtheorem{rem}[Teo]{Remark}
\newtheorem{cor}[Teo]{Corollary}
\newtheorem{Prop}[Teo]{Proposition}

\newtheorem{theorem}{Theorem}[section]

\DeclareRobustCommand{\gobblefour}[4]{}
\begin{document}

\title[Regularity for an anisotropic equation in the plane]
{Regularity for an anisotropic equation in the plane}
%\begin{comment}
\author{Peter Lindqvist}
\address{Peter Lindqvist,  Norwegian University of Science and Technology, Department of Mathematics, N-7491 Trondheim, Norway}
\email{peter.lindqvist@ntnu.no}
\author{Diego Ricciotti}
\address{Diego Ricciotti, University of Pittsburgh, Department of Mathematics, 
301 Thackeray Hall, Pittsburgh, PA 15260, USA}
\email{DIR17@pitt.edu}

%\end{comment}
\begin{abstract}
We present a simple proof of the $C^1$ regularity of $p$-anisotropic  functions in the plane for $2\leq p<\infty$. We achieve a logarithmic modulus of continuity for the derivatives. The monotonicity (in the sense of Lebesgue) of the derivatives is used. The case with two exponents is also included.
\end{abstract}
\maketitle
\tableofcontents

\begin{comment}
\let\thefootnote\relax\footnote{2010 Mathematics Subject Classification: $35$H$20$, $35$J$70$.}
\let\thefootnote\relax\footnote{{\it Keywords}: $p$-Orthotropic Equation, Regularity.}
\end{comment}

\section{Introduction}

The minimization of the ``anisotropic" variational integral
\begin{equation}\label{anisotropicIntegral}
I_{\Omega}(v)=\int_{\Omega}\sum_{i=1}^n\frac{1}{p_i}\left\lvert \frac{\partial v}{\partial x_i} \right\rvert^{p_i}\!\dx
\end{equation}
over functions $v(x)=v(x_1,\cdots,x_n)$ with given values on the boundary of the bounded domain $\Omega\subset\mathbb{R}^n$, leads to the Euler-Lagrange equation
\begin{equation}\label{anisotropicEquation}
\int_{\Omega}\sum_{i=1}^n\left\lvert \frac{\partial u}{\partial x_i} \right\rvert^{p_i-2}\frac{\partial u}{\partial x_i}\frac{\partial \phi}{\partial x_i}\dx =0
\end{equation} 
for all test functions $\phi\in C_0^\infty(\Omega)$. 
%It is required that the Sobolev derivatives $\frac{\partial u}{\partial x_i}$ belong to $L^{p_i}_{loc}(\Omega)$ for $i=1,2,\cdots,n$.
Denoting by $\p=(p_1,\cdots,p_n)$, it is required that a solution $u$ belongs to the anisotropic Sobolev space
$$ W^{1,\p}(\Omega) :=\left\{ u\in W^{1,1}(\Omega)\;:\;u_{x_i}\in L^{p_i}(\Omega)\;,\; i=1,\cdots,n \right\}. $$
Formally one has the equation
\begin{equation*}
\sum_{i=1}^n \frac{\partial}{\partial x_i}\left( \left\lvert \frac{\partial u}{\partial x_i} \right\rvert^{p_i-2}\!\frac{\partial u}{\partial x_i}\right)=0
\end{equation*}
in $\Omega$.

The equation is demanding even in the plane. We shall restrict ourselves to the case $n=2$ and $2\leq p_1\leq p_2< \infty$.
Our object is the continuity of the gradient $\nabla u =(u_{x_1}, u_{x_2})$ in the plane.
The recent work \cite{BB} of P. Bousquet and L. Brasco is devoted to the ``orthotropic equation", as they call it, 
\begin{equation}\label{pOrthotropic}
\frac{\partial}{\partial x_1}\left( \left\lvert\frac{\partial u}{\partial x_1}\right\rvert^{p-2}\!\frac{\partial u}{\partial x_1} \right)
+\frac{\partial}{\partial x_2}\left( \left\lvert\frac{\partial u}{\partial x_2}\right\rvert^{p-2}\!\frac{\partial u}{\partial x_2} \right)
=0
\end{equation}
in $\Omega$, with only one exponent $1<p<\infty$. They proved that $u\in C^1_{loc}(\Omega)$.
Our first result is a very simple proof of the continuity of the gradient.

\begin{theorem}\label{Theorem1}
Let $p\geq 2$ and suppose that $u\in W^{1,p}(\Omega)$ is a solution of \eqref{pOrthotropic} in $\Omega$. Then $\nabla u$ is continuous and
\begin{equation*}
\osc{B_r}(\nabla u) \leq A \left( \frac{1}{R^2\log\left(\frac{R}{r}\right)}
\iint_{B_{2R}}|\nabla u|^p \dxy\right)^\frac{1}{p}
\end{equation*} 
where $A=A(p)$ and $B_r$, $B_R$ are concentric balls $B_r\subset B_R\subset B_{2R}\subset\subset\Omega$.
\end{theorem} 

The advantage of our proof is, besides its simplicity, that a modulus of continuity of the size
$$\left(\log\left(\frac{1}{r}\right)\right)^{-\frac{1}{p}}$$ 
is provided. 
The main ingredient is an elementary inequality used by Lebesgue in 1907, valid for functions that are monotone (in the sense of Lebesgue). We exploit the fact that the partial derivative $u_{x_i}$ obeys the maximum and minimum principle, a key property observed in \cite{BB} Lemma 2.6 .

Second, we consider the equation
\begin{equation}
\frac{\partial}{\partial x_1}\left( \left\lvert\frac{\partial u}{\partial x_1}\right\rvert^{p_1-2}\!\frac{\partial u}{\partial x_1} \right)
+\frac{\partial}{\partial x_2}\left( \left\lvert\frac{\partial u}{\partial x_2}\right\rvert^{p_2-2}\!\frac{\partial u}{\partial x_2} \right)
=0
\end{equation}
in $\Omega$, under the restriction $2\leq p_1<p_2$.

\begin{theorem}\label{Theorem2}
Let $\p=(p_1,p_2)$ with $2\leq p_1<p_2$ and assume that 
%$\frac{\partial u}{\partial x_1}\in L^{p_1}(\Omega)$, $\frac{\partial u}{\partial x_2}\in L^{p_2}(\Omega)$. 
$u\in W^{1,\p}(\Omega)$. 
If $u=u(x_1,x_2)$ is a solution of equation \eqref{anisotropicEquation}, then the gradient $\nabla u$ is continuous and
\begin{equation}\label{estimate2}
\osc{B_r}\left( u_{x_i}\right)
\leq A\left(\frac{1}{R^2\log(R/r)}\iint_{B_{2R}}(|\nabla u|^{p_1}+|\nabla u|^{p_2})\dxy\right)^\frac{1}{p_i},
\end{equation}
where $A=A(p_1,p_2)$ and $B_r$, $B_R$ are concentric balls $B_r\subset B_R\subset B_{2R}\subset\subset\Omega$. The integral converges.
\end{theorem}
Here we encounter an extra difficulty. Naturally 
$ u_{x_1}\in L^{p_1}(\Omega)$, 
$ u_{x_2}\in L^{p_2}(\Omega)$,
and, consequently, $u_{x_2}\in L^{p_1}(\Omega)$, 
but one cannot assume
$u_{x_1}\in L^{p_2}(\Omega)$.
Indeed, the term 
$|v_{x_1}|^{p_2}$
is not present in the variational integral \eqref{anisotropicIntegral}.
This difficulty is discussed in \cite{Leonetti}. Under the restriction $p_2<p_1+2$, this problem is settled in Proposition \ref{PropositionBeta} below, the proof of which is a direct adaptation of the method in \cite{ELM}. By a recent result in \cite{BLPV}, the solution of \eqref{anisotropicEquation} is locally Lipschitz continuous ($n=2,\,\, 2\leq p_1\leq p_2$), see  Theorem 1.4 and Remark 1.5 there. By Rademacher's Theorem the gradient belongs to $L^{\infty}_{loc}(\Omega).$
Thus the integral in the right hand side is convergent also for $p_2 \geq p_1+2.$ Nonetheless, we have included a sketch of the proof based on the iteration in \cite{ELM}, since the extra assumption leads to a considerable simplification. Furthermore, this approach seems to allow a generalization to the vector valued case.

\subsection{Acknowledgments} We thank Lorenzo Brasco for a discussion about \cite{BLPV}.
P.L. was partially supported by the Norwegian Research Council. He wants to thank the University of Pittsburgh for its hospitality.

\section{Standard Estimates}
\paragraph{\bf Notation.}

We use standard notation.
$B_r=B_r(a)$ denotes the ball $\{x\in\mathbb{R}^2 \,:\, |x-a|<r\}$ and when several balls like $B_r$, $B_R$ appear in the same formula they are assumed to be concentric.
Usually, $\sum_i$ means $\sum_{i=1}^2$, although the formulas in this section are valid also in $n$ dimensions.
A variable subscript in a function denotes a derivative with respect to that variable, e.g. $v_{x_i}=\frac{\partial v}{\partial x_i}$ and $v_{x_ix_j}=\frac{\partial^2 v}{\partial x_i\partial x_j}$.\\

\medskip
\paragraph{\bf Regularization.}

We shall regularize the equation so that at least second continuous derivatives are available. 
The variational integral
\begin{equation*}
I^\epsilon_{\Omega}(v)=\sum_i \iint_{\Omega} \left(\frac{|v_{x_i}|^{p_i}}{p_i}+\epsilon(p_i-2) \frac{v_{x_i}^2}{2}\right) \dxy \quad \epsilon >0,
\end{equation*}
has Euler-Lagrange equation
\begin{equation}\label{anisotropicEquationRegularized}
\sum_i \iint_{\Omega} \left(|u_{x_i}|^{p_i-2}u_{x_i}+\epsilon(p_i-1)u_{x_i}\right) \phi_{x_i}\dxy=0
\end{equation}
valid for all $\phi\in C_0^\infty(\Omega)$. Let $u^\epsilon\in W^{1,\p}(\Omega)$ denote a solution. By elliptic regularity theory, $u^\epsilon$ is smooth.\\

\medskip

\paragraph{\bf Estimates.}

Below $\xi\in C_0^\infty(\Omega)$ is a test function, $0\leq \xi \leq 1$.
Recall that $p_1\leq p_2$.

%\begin{Lemma}\label{LemmaADD1}
%Let $B_R\subset\subset\Omega$, $u$ be a solution of \eqref{anisotropicEquation} and $u^\epsilon$ the solution of \eqref{anisotropicEquationRegularized} with boundary values $u$ on $\partial B_R$. Then we have
%\begin{equation*}
%\iint_{B_R} |u^\epsilon_{x_\nu}|^{p_\nu}
%\leq C(p_1,p_2) \sum_i \iint_{B_R}|u_{x_i}|^{p_i}\dxy+\epsilon(p_2-1)\iint_{B_R}|\nabla u|^2\dxy.
%\end{equation*}
%for $\nu=1$, $2$.
%\end{Lemma}
%\begin{proof}
%In equation \eqref{anisotropicEquationRegularized} use the test function $\phi=u^\epsilon-u$ and Young's inequality.
%\end{proof}

\begin{Lemma}
Let $u^\epsilon$ be a solution of \eqref{anisotropicEquationRegularized}.
We have
\begin{equation*}
\begin{split}
\sum_i \iint_{\Omega} \xi^{p_2} |u^\epsilon_{x_i}|^{p_i}\dxy
\leq a&\sum_i \iint_{\Omega} \xi^{p_2-p_i}|\xi_{x_i}|^{p_i}|u^\epsilon|^{p_i}\dxy\\
&+\epsilon(p_2-1)p_2^2\iint_{\Omega}\xi^{p_2-2}|\nabla\xi|^2|u^\epsilon|^2\dxy,
\end{split}
\end{equation*}
where $a=a(p_1,p_2)$.
\end{Lemma} 
\begin{proof}
Use the test function $\phi=\xi^{p_2}u^\epsilon$ in \eqref{anisotropicEquationRegularized}.
\end{proof}

\begin{Lemma}
Let $u^\epsilon$ be a solution of \eqref{anisotropicEquationRegularized}.
For $\nu=1$, $2$ we have
\begin{equation*}
\begin{split}
\sum_i \iint_{\Omega}(p_i-1)\xi^2|u^\epsilon_{x_i}|^{p_i-2}(u^\epsilon_{x_ix_\nu})^2\dxy
&\leq 4\sum_i \iint_{\Omega}(p_i-1)\xi_{x_i}^2|u_{x_i}^{\epsilon}|^{p_i-2}(u^\epsilon_{x_\nu})^2\dxy\\
&\quad+ 4\epsilon(p_2-1) \iint_{\Omega}|\nabla\xi|^2(u^\epsilon_{x_\nu})^2\dxy.
\end{split}
\end{equation*}
\end{Lemma}

\begin{proof}
We can use the derivative $\phi_{x_\nu}$ in place of $\phi$ as a test function in \eqref{anisotropicEquationRegularized}.
An integration by parts with respect to $x_\nu$ yields the differentiated equation
\begin{equation}\label{differentiatedEquation}
\sum_i \iint_{\Omega}(p_i-1)(|u^\epsilon_{x_i}|^{p_i-2}+\epsilon)u^\epsilon_{x_ix_\nu}\phi_{x_i}\dxy
=0.
\end{equation}
Now use the test function
\begin{equation*}
\begin{split}
\phi &= \xi^2u^\epsilon_{x_\nu}\\
\phi_{x_i}&=\xi^2u^\epsilon_{x_ix_\nu}+2\xi\xi_{x_i}u^\epsilon_{x_\nu}
\end{split}
\end{equation*}
and Young's inequality.
\end{proof}

\begin{rem}
The quantity 
$|u^\epsilon_{x_2}|^{p_2-2}(u^\epsilon_{x_1})^2$
has unfavourable exponents. A bound independent of $\epsilon$ is not immediate for the term
$$\iint_{\Omega}\xi_{x_2}^2|u^\epsilon_{x_2}|^{p_2-2}(u^\epsilon_{x_1})^2\dxy.$$
\end{rem}

\begin{cor}\label{Corollary}
Let $u^\epsilon$ be a solution of \eqref{anisotropicEquationRegularized}.
We have
\begin{equation*}
\begin{split}
\sum_i \iint_{\Omega} \xi^2 \left\lvert  \nabla\left( |u^\epsilon_{x_i}|^\frac{p_i-2}{2}u^\epsilon_{x_i} \right)  \right\rvert^2\dxy
\leq C \Bigg( \sum_i \iint_{\Omega}|\nabla\xi|^2 |u^\epsilon_{x_i}|^{p_i-2}|\nabla u^\epsilon|^2&\dxy\\
+\epsilon\iint_{\Omega}|\nabla\xi|^2|\nabla u^\epsilon|^2&\dxy\Bigg)
\end{split}
\end{equation*}
where $C=C(p_1,p_2)$.
\end{cor}
\begin{proof}
Use
$$\abs{ \frac{\partial}{\partial x_\nu}\left( |u^\epsilon_{x_i}|^\frac{p_i-2}{2}u^\epsilon_{x_i} \right) }^2
= \left(\frac{p_i}{2}\right)^2|u^\epsilon_{x_i}|^{p_i-2}(u^\epsilon_{x_ix_\nu})^2$$
and sum over $\nu$.
\end{proof}

\medskip

\paragraph{\bf Convergence }$u^\epsilon\longrightarrow u$.\\

Let $u\in W^{1,\p}(\Omega)$ be a solution of equation \eqref{anisotropicEquation}. 
Here we take $B_R\subset\subset\Omega$ and let $u^\epsilon$ be the solution of \eqref{anisotropicEquationRegularized}
%$$ \sum_i \frac{\partial}{\partial x_i}\left( |u^\epsilon_{x_i}|^{p_i-2}u^\epsilon_{x_i}+\epsilon(p_i-1)u^\epsilon_{x_i} \right)=0 $$
with boundary values $u$ on $\partial B_R$.
Subtract the weak equations \eqref{anisotropicEquation} and \eqref{anisotropicEquationRegularized} and use the test function $\phi=u^\epsilon-u$.
After some arrangements
\begin{equation*}
\begin{split}
\sum_i&\iint_{B_R} ( |u^\epsilon_{x_i}|^{p_i-2}u^\epsilon_{x_i}-|u_{x_i}|^{p_i-2}u_{x_i} )
(u^\epsilon_{x_i}-u_{x_i})\dxy\\
&+\sum_i \epsilon(p_i-1)\iint_{B_R} (u^\epsilon_{x_i}-u_{x_i})^2\dxy\\
&=  \sum_i \epsilon(p_i-1)\iint_{B_R} u_{x_i}(u^\epsilon_{x_i}-u_{x_i})\dxy\\
&\leq \frac{\epsilon}{2}\sum_i(p_i-1)\iint_{B_R}u_{x_i}^2\dxy
+\frac{\epsilon}{2}\sum_i(p_i-1)\iint_{B_R}(u^\epsilon_{x_i}-u_{x_i})^2\dxy
\end{split}
\end{equation*}
and the last term can be absorbed into the left-hand side.
The inequality
\begin{equation}\label{pInequality}
2^{2-p}|b-a|^p\leq (|b|^{p-2}b-|a|^{p-2}a)(b-a)
\end{equation}
yields
$$ \sum_i 2^{2-p_i}\iint_{B_R}|u^\epsilon_{x_i}-u_{x_i}|^{p_i}\dxy
\leq \frac{\epsilon}{2}(p_2-1)\iint_{B_R}|\nabla u|^2\dxy. $$
The next Lemma follows from this.
\begin{Lemma}\label{gradconv}
Assume $u\in W^{1,\p}(\Omega)$ solves equation \eqref{anisotropicEquation} and let $u^\epsilon$ be the solution of \eqref{anisotropicEquationRegularized} in $B_R$ with boundary values $u$ on $\partial B_R$. Then
\begin{equation*}
\begin{split}
&u^\epsilon \longrightarrow u \;\;\text{uniformly in}\;\; B_R\\
&u^\epsilon_{x_i} \longrightarrow u_{x_i} \;\;\text{in}\;\; L^{p_i}(B_R)
\end{split}
\end{equation*}
as $\epsilon\to 0$.
\end{Lemma}

\begin{proof}
  It remains to establish the convergence  of the functions.
If $p_1>2$ it follows from Morrey's inequality in the plane that $$u^\epsilon\longrightarrow u \;\;\text{uniformly in}\;\; B_R.$$
The case $p_1=2$ follows from Lemma \ref{LemmaLebesgue} below, since the maximum/minimum principle obviously is valid for $u^\epsilon$. 
\end{proof}

\section{Oscillation of monotone functions}
A continuous function $v:\Omega\longrightarrow\mathbb{R}$ is monotone (in the sense of Lebesgue) if
$$\max_{\overline{D}}v=\max_{\partial D}v\quad \text{and}\quad  \min_{\overline{D}}v=\min_{\partial D}v$$
for all subdomains $D\subset\subset \Omega$. For the next Lemma it us enough that
$$\osc{B_r}v=\osc{\partial B_r}v$$
holds for circles. Monotone functions are discussed in \cite{Manf}.
%NOTE We need Leb only for cont functions.
\begin{Lemma}[Lebesgue]\label{LemmaLebesgue}
Let $\Omega\subset\mathbb{R}^2$. If $v\in W^{1,2}_{loc}(\Omega)\cap C(\Omega)$ is monotone, then
\begin{equation}\label{LebesgueEstimate}
(\osc{B_{r_1}}v)^2
\log\left(\frac{r_2}{r_1}\right)
\leq \pi \iint_{B_{r_2}}|\nabla v|^2\dxy
\end{equation}
holds for all concentric disks $B_{r_1}\subset B_{r_2}\subset\subset \Omega$.
\end{Lemma}

\begin{proof}
As on page 388 of \cite{L} an integration in polar coordinates yields
$$v(r,\theta_2)-v(r,\theta_1)=\int_{\theta_1}^{\theta_2} \frac{\partial v(r,\theta)}{\partial\theta}\,\mbox{d}\theta$$
for a smooth function $v$. 
It is enough to integrate over a half circle and use the Cauchy-Schwartz inequality to obtain
$$ (\osc{\partial B_r} v )^2 \leq \pi \int_0^{2\pi}\abs{\frac{\partial v}{\partial\theta}}^2\,\mbox{d}\theta. $$
Since
$$|\nabla v|^2=\left(\frac{\partial v}{\partial r}\right)^2+\frac{1}{r^2}\left(\frac{\partial v}{\partial \theta}\right)^2
\geq \frac{1}{r^2}\left(\frac{\partial v}{\partial \theta}\right)^2$$
we have
$$\frac{1}{r}(\osc{\partial B_r}v)^2
\leq \pi \int_0^{2\pi}|\nabla v|^2r\,\mbox{d}\theta$$
integrated over a circle of radius $r$.
By the monotonicity
$$\osc{\partial B_r}v=\osc{B_r}v\geq \osc{B_{r_1}}v$$
when $r\geq r_1$.
An integration with respect to $r$ yields \eqref{LebesgueEstimate}.
The Lemma follows by approximation.
\end{proof}

We shall apply the oscillation Lemma to the functions
$$|u^\epsilon_{x_i}|^\frac{p_i-2}{2}u^\epsilon_{x_i}.$$
To this end, we prove that $u^\epsilon_{x_i}$ is monotone. This is credited to \cite{BB}.

\begin{Prop}
Let $u^\epsilon$ denote a solution of equation \eqref{anisotropicEquationRegularized}.
Then
$$\min_{\partial B_r}u^\epsilon_{x_\nu}\leq u^\epsilon_{x_\nu}(x)\leq \max_{\partial B_r}u^\epsilon_{x_\nu}$$
when $x\in B_r$, $B_r\subset\subset\Omega$, and $\nu=1$, $2$.
\end{Prop}
\begin{proof}
Fix $\nu$.
Assume first that $u^\epsilon_{x_\nu}\leq C$ on $\partial B_r$, where $C$ is a constant. We claim that $u^\epsilon_{x_\nu}\leq C$ in $B_r$.

Use the test function
$$\phi^+(x)=(u^\epsilon_{x_\nu}-C)^+=\max\{u^\epsilon_{x_\nu}-C,\, 0\}$$
defined in $\overline{B_r}$.
Note that $\phi^+=0$ on $\partial B_r$ and set $\phi^+=0$ outside $B_r$.
Then $\phi^+$ is admissible in the differentiated equation \eqref{differentiatedEquation}. 
It follows that
\begin{equation*}
\begin{split}
0&=\iint_{B_r}(p_i-1)(|u^\epsilon_{x_i}|^{p_i-2}+\epsilon)u^\epsilon_{x_ix_\nu}(u^\epsilon_{x_\nu}-C)^+_{x_i}\dxy\\
&\geq \epsilon\sum_i(p_i-1)\iint_{B_r}|(u^\epsilon_{x_\nu}-C)^+_{x_i}|^2\dxy
\end{split}
\end{equation*}
and hence
$$(u^\epsilon_{x_\nu}-C)^+_{x_i}=0 \;\;\text{in}\;\;B_r$$
for $i=1$, $2$.
Thus $(u^\epsilon_{x_\nu}-C)^+$ is constant in $B_r$. We conclude that 
$u^\epsilon_{x_\nu}\leq C$ in $B_r$ as desired.

A similar proof goes for the case $u^\epsilon_{x_\nu}\geq C$. Now use 
$$\phi^-(x)=(C-u^\epsilon_{x_\nu})^+.$$
\end{proof}

\begin{cor}
Let $u^\epsilon$ denote a solution of equation \eqref{anisotropicEquationRegularized}.
For $i=1$, $2$ the function
$$|u^\epsilon_{x_i}|^\frac{p_i-2}{2}u^\epsilon_{x_i}$$
is monotone in $\Omega$.
\end{cor}

\section{The case $p_1=p_2\geq 2$}

Let $p_1=p_2=p\geq 2$ and let $u\in W^{1,p}(\Omega)$ be a solution of \eqref{anisotropicEquation}.
In order to prove Theorem \ref{Theorem1} we denote by $u^\epsilon$ the solution  of the regularized equation \eqref{anisotropicEquationRegularized} in $B_{2R} \subset\subset\Omega$ with boundary values $u^\epsilon=u$ on $\partial B_{2R}$.
Let $r\leq R$. By Lebesgue's Lemma
$$\osc{B_r}^2\left( |u^\epsilon_{x_i}|^\frac{p-2}{2}u^\epsilon_{x_i} \right)\log\left(\frac{R}{r}\right)
\leq \pi \iint_{B_R}|\nabla(|u^\epsilon_{x_i}|^\frac{p-2}{2}u^\epsilon_{x_i})|^2\dxy.$$
Observe that 
$$2^{2-p}(\osc{B_r}u^\epsilon_{x_i})^p\log\left(\frac{R}{r}\right)
\leq \osc{B_r}^2\left( |u^\epsilon_{x_i}|^\frac{p-2}{2}u^\epsilon_{x_i} \right)\log\left(\frac{R}{r}\right)$$
by the elementary inequality \eqref{pInequality}.
Choose the test function $\xi$ in Corollary \ref{Corollary} so that $0\leq \xi\leq 1$, $\xi=1$ in $B_R$, $\xi=0$ in $\Omega\setminus B_{3R/2}$, and $|\nabla\xi|\leq CR^{-1}$. Thus we can majorize the right hand side:
\begin{equation*}
\begin{split}
\iint_{B_R} |\nabla(|u^\epsilon_{x_i}|^\frac{p-2}{2}u^\epsilon_{x_i})|^2 \dxy
\leq \frac{C_p'}{R^2}\left( \iint_{B_{2R}}|\nabla u^\epsilon|^p\dxy
+\epsilon\iint_{B_{2R}}|\nabla u^\epsilon|^2\dxy \right)
\end{split}
\end{equation*}
which is uniformly bounded in $\epsilon$ ($0<\epsilon<1$). To see this, it is enough to test equation \eqref{anisotropicEquationRegularized} with $\phi=u^\epsilon-u$ and use Young's inequality to get
\begin{equation*}
\sum_i \iint_{B_{2R}}  |u^\epsilon_{x_i}|^{p}
\leq C(p)\iint_{B_{2R}}|\nabla u|^{p}\dxy+\epsilon(p-1)\iint_{B_{2R}}|\nabla u|^2\dxy.
\end{equation*}
% according to Lemma \ref{LemmaADD1}.

Since, by Lemma \ref{gradconv}, $u^\epsilon_{x_i}\longrightarrow u_{x_i}$ a.e. in $B_{2R}$ as $\epsilon\to 0$ (at least for a subsequence), we finally obtain
$$(\osc{B_r}u_{x_i})^p\log\left(\frac{R}{r}\right)
\leq \frac{C_p}{R^2}\iint_{B_{2R}}|\nabla u|^p\dxy.$$
This concludes the proof of Theorem \ref{Theorem1}.

\begin{rem}
For $B_{4R}\subset\subset\Omega$ let  $\xi\in C_0^\infty(B_{4R})$, $0\leq \xi\leq 1$, $\xi=1$ in $B_{2R}$, and $|\nabla\xi|\leq CR^{-1}$.
Testing equation \eqref{anisotropicEquation} with $\phi=u\xi^{p_2}$  and using Young's inequality we obtain
\begin{equation*}
\sum_i \iint_{B_{4R}}\xi^{p_2}|u_{x_i}|^{p_i}\dxy
\leq C(p_1,p_2) \sum_i\iint_{B_{4R}} \xi^{p_2-p_i}|\nabla\xi|^{p_i}\,|u|^{p_i}\dxy.
\end{equation*}
Hence we can write the estimate of Theorem \ref{Theorem1} in the form
$$(\osc{B_r}u_{x_i})^p \leq \frac{D_p}{R^{2+p}\log\left(\frac{R}{r}\right)}
\iint_{B_{4R}}|u|^p\dxy$$
for $r<R$.
\end{rem}

\section{The case $2\leq p_1<p_2<p_1+2$}
We shall adapt the proof in \cite{ELM} to obtain the following summability result for the derivative of the solution $u^\epsilon$ of the regularized equation \eqref{anisotropicEquationRegularized}. 
\footnote{The assumption $|z|^{p_1}\leq f(z)$ on page $417$, eqn. (2.2) of \cite{ELM} is not valid here, but we have $|z_1|^{p_1}+|z_2|^{p_2}\leq f(z)$ instead.}
We omit the details and refer to \cite{ELM} for missing parts.

\begin{Prop}\label{PropositionBeta}
Let $B_R\subset\subset\Omega$ and let $u^\epsilon$ be a solution of equation \eqref{anisotropicEquationRegularized} in $B_R$.
Then there exists an exponent $\beta=\beta(p_1,p_2)$ and a constant $C=C(p_1,p_2,r,R)$ such that 
\begin{equation}\label{inequalityP2}
\iint_{B_r}|u^\epsilon_{x_1}|^{p_2}\dxy
\leq C \left( \iint_{B_R}(1+|u^\epsilon_{x_1}|^{p_1}+|u^\epsilon_{x_2}|^{p_2})\dxy \right)^\beta
\end{equation}
for all $r<R$.
\end{Prop}
The proof is based on a double regularization. The Euler-Lagrange equation of the variational integral
\begin{equation*}
I^{\epsilon,\sigma}_{B_R}(v)
=\sum_i\iint_{B_R}\left(\frac{|v_{x_i}|^{p_i}}{p_i}+\epsilon(p_i-1)\frac{|v_{x_i}|^2}{2}\right)\dxy
+\sigma \iint_{B_R}\frac{|v_{x_1}|^{p_2}}{p_2}\dxy
\end{equation*}
is
\begin{equation}\label{EquationDoubleRegularized} \iint_{B_R}\sum_i(|u_{x_i}|^{p_i-2}+\epsilon(p_i-1))u_{x_i}\phi_{x_i} +\sigma|u_{x_1}|^{p_2-2}u_{x_1}\phi_{x_1} \dxy=0,
\end{equation}
for all $\phi\in C_0^\infty(B_R)$.
Let $u^{\epsilon, \sigma}$ denote the solution of \eqref{EquationDoubleRegularized} with boundary values $u^{\epsilon, \sigma}=u^\epsilon$ on $\partial B_R$.
A similar reasoning as in \cite{ELM}, pages 421-427,
leads for any $\delta$, $p_1\leq \delta<p_2$, to the estimate 
\begin{equation*}
\left( \iint_{B_{\alpha^2R}}|u^{\epsilon, \sigma}_{x_1}|^\frac{p_1}{1-b} \dxy\right)^\frac{1-b}{2}
\leq C(p_1,p_2,\delta,R,\alpha)
\left( \left(\iint_{B_{\alpha R}} |u^{\epsilon, \sigma}_{x_1}|^{p_1}\dxy\right)^\frac{1}{2}
+ \left(\iint_{B_{\alpha R}} |\nabla u^{\epsilon, \sigma}|^{\delta}\dxy\right)^\frac{1}{2}\right)
\end{equation*}
where $0<\alpha<1$.
This is valid for every $b$ in the range
$$0<b<2-\frac{p_2}{\delta}.$$
The idea is to iterate this estimate over concentric disks of radii $\alpha R$, $\alpha^2R$, $\alpha^3R$, ...( a finite number will do) starting with $\delta_0=p_1$ and increasing the exponent at each step.
If, for instance, 
$$2\kappa = \frac{2p_1-p_2}{p_2-p_1}p_1$$
we can always find an admissible $b$ such that
$$\frac{p_1}{1-b}=\delta+\kappa.$$
Hence the powers in the iteration become $p_1$, $p_1+\kappa$, $(p_1+\kappa)+\kappa=p_1+2\kappa$, ..., $p_1+m\kappa$.
This yields the Lemma, but for $u^{\epsilon,\sigma}$ insted of $u^\epsilon$.
The limit procedure $\sigma\to 0$ leads to the desired result, when one uses the minimization property
$$ I^{\epsilon, \sigma}_{B_R}(u^{\epsilon,\sigma})
\leq I^{\epsilon}_{B_R}(u^\epsilon)+\frac{\sigma}{2}\iint_{B_R}|u^\epsilon_{x_1}|^{p_2}\dxy ,$$
provided that we already know
\begin{equation}\label{difficultTerm}
\iint_{B_R}|u^\epsilon_{x_1}|^{p_2}\dxy <\infty.
\end{equation}
To get rid of this restriction, we use a convenient convolution with some mollifier $\rho$:
$$u^*=u*\rho$$
approximating the solution $u$ of the original equation \eqref{anisotropicEquation}.
Let $u^{\epsilon,*}$ be the solution of the regularized equation \eqref{anisotropicEquationRegularized} in $B_R$, with boundary values $u^{\epsilon,*}=u^*$ on $\partial B_R$. 
Let $u^{\epsilon,\sigma,*}$ be a solution of \eqref{EquationDoubleRegularized} with the same boundary values.
Then the difficult term \eqref{difficultTerm} can be dismissed, since now
\begin{equation*}
\begin{split}
I^{\epsilon,\sigma}_{B_R}(u^{\epsilon,\sigma,*})
&\leq I^{\epsilon}_{B_R}(u^*)+\frac{\sigma}{2}\iint_{B_R}|u^*_{x_1}|^{p_2}\dxy\\
&\leq I^{\epsilon}_{B_{R^*}}(u)+\frac{\sigma}{2}\iint_{B_{R^*}}|u_{x_1}|^{p_2}\dxy
\end{split}
\end{equation*}
where $R^*>R$, and $R^*-R$ can be made as small as we please (depending on $\rho$ in the convolution). We used the fact that the convolution is a contraction.
We now have a bound free of $\sigma$ and can take the limit as $\sigma\to 0$.
The result is
$$ \iint_{B_R}|u^{\epsilon,*}_{x_1}|^{p_2}\dxy
\leq C\left( \iint_{B_{R^*}}(1+|u_{x_1}|^{p_1}+|u_{x_2}|^{p_2}) \dxy\right)^\beta. $$
As $u^*=u*\rho \to u$, we conclude from
$$I_{B_R}(u^{\epsilon,*})\leq I^{\epsilon}_{B_R}(u^{\epsilon,*}) \leq I^{\epsilon}_{B_R}(u^*)\to I^{\epsilon}_{B_R}(u)$$
that the weak limit in $L^2(B_R)$ of $\nabla u^{\epsilon,*}$ must be $\nabla u^\epsilon$, since the minimizer of this strictly convex variational integral is unique.
By weak lower semicontinuity
$$\iint_{B_r}|u^{\epsilon}_{x_1}|^{p_2}\dxy
\leq C\left( \iint_{B_R}(1+|u_{x_1}|^{p_1}+|u_{x_2}|^{p_2}) \dxy\right)^\beta$$
since $R^*\to R$.
This version of inequality \eqref{inequalityP2} is enough for us.

\section{Proof of Theorem \ref{Theorem2}}
\begin{proof}[Proof of Theorem \ref{Theorem2}]
This is almost the same as the proof of Theorem \ref{Theorem1}. 
For the regularized equation \eqref{anisotropicEquationRegularized} one first obtains as before, though with a few obvious changes,
$$R^2(\osc{B_r}u^\epsilon_{x_i})^{p_i}\log\left(\frac{R}{r}\right)
\leq A \iint_{B_{2R}}(|\nabla u^\epsilon|^{p_1}+|\nabla u^\epsilon|^{p_2})\dxy$$
when $r\leq R$, $B_{2R}\subset\subset\Omega$ and for a constant $A=A(p_1,p_2)$.

In the case $p_2 < p_1 + 2$ we proceed as follows. 
If $B_{4R}\subset\subset\Omega$, using Proposition \ref{PropositionBeta},  we can bound the right hand side by
$$C(p_1,p_2,R) \left( \iint_{B_{4R}}\left(1+|u_{x_1}|^{p_1}+|u_{x_2}|^{p_2}\right)\dxy \right)^{\beta(p_1,p_2)}$$ 
which is a finite quantity independent of $\epsilon$. The desired result follows as $\epsilon\to 0$.

 The general case can be extracted from Theorem 1.4 in \cite{BLPV}.
\end{proof}

\begin{rem}
The exact dependence on $R$ is not worked out here. The result that comes from the iteration in the proof of Proposition \eqref{PropositionBeta} above, if all steps are computed, is not illuminating.
\end{rem}

\bibliography{reference_bib}\nocite{*}

\begin{thebibliography}{1}

\bibitem{BB}
P.~Bousquet and L.~Brasco.
\newblock {$C^1$} regularity of orthotropic {$p$}-harmonic functions in the
  plane.
\newblock {\em Preprint}, 2016.

\bibitem{BBJ}
P.~Bousquet, L.~Brasco, and V.~Julin.
\newblock Lipschitz regularity for local minimizers of some widely degenerate
  problems.
\newblock {\em Annali della Scuola Normale Superiore di Pisa, Classe di
  Scienze}, 2016.

\bibitem{BLPV}
L.~Brasco, C.~Leone, G.~Pisante, and A.~Verde.
\newblock Sobolev and {L}ipschitz regularity for local minimizers of widely
  degenerate anisotropic functionals.
\newblock {\em Nonlinear Anal.}, 153:169--199, 2017.

\bibitem{ELM}
L.~Esposito, F.~Leonetti, and G.~Mingione.
\newblock Higher integrability for minimizers of integral functionals with
  {$(p,q)$} growth.
\newblock {\em J. Differential Equations}, 157(2):414--438, 1999.

\bibitem{L}
H.~Lebesgue.
\newblock Sur le probl{\`e}me de dirichlet.
\newblock {\em RENDICONTI DEL CIRCOLO MATEMATICO DI PALERMO}, 24:371--402,
  1907.

\bibitem{Leonetti}
F.~Leonetti.
\newblock Higher differentiability for weak solutions of elliptic systems with
  nonstandard growth conditions.
\newblock {\em Ricerche Mat.}, 42(1):101--122, 1993.

\bibitem{Manf}
J.~J. Manfredi.
\newblock Weakly monotone functions.
\newblock {\em J. Geom. Anal.}, 4(3):393--402, 1994.

\end{thebibliography}
\bibliographystyle{abbrv}

\end{document}